\documentclass[leqno,11pt]{amsart}
\usepackage{amssymb}
\newtheorem{theorem}{Theorem}[section]
\newtheorem{proposition}[theorem]{Proposition}
\newtheorem{lemma}[theorem]{Lemma}
\newtheorem{corollary}[theorem]{Corollary}

\usepackage{comment}

\theoremstyle{definition}

\newtheorem{remark}[theorem]{Remark}
\usepackage{hyperref}
\usepackage{ stmaryrd }
\usepackage{color}
\usepackage[shortlabels]{enumitem}

\newcommand{\p}{\partial}

\newcommand{\Ric}{\mathrm{Ric}}

\newcommand\eea{\end{eqnarray}}
\newcommand\bea{\begin{eqnarray}}

\def\be{\begin{equation}}
\def\ee{\end{equation}}

\title[Mean curvature flow in de Sitter space]{Mean Curvature Flow  in de Sitter space}
\author{Or Hershkovits, Leonardo Senatore}
\begin{document}

\begin{abstract}
We study mean convex mean curvature flow $M_s$ of local spacelike graphs in the flat slicing of de Sitter space. We show that if the initial slice is  of non-negative time  and is graphical over a large enough ball,  and if $M_s$ is of  bounded mean curvature, then as $s$ goes to infinity, $M_s$ becomes graphical  in \textit{expanding} balls, over which the gradient function converges to $1$. In particular, if $p_s$ is the point lying over the center of the domain ball in $M_s$, then $(M_s,p_s)$ converges smoothly  to the flat slicing of de Sitter space. This has some relation to the mean curvature flow approach to the cosmic no hair conjecture.
\end{abstract}

\maketitle

\section{introduction}

The study of evolution of graphs by mean curvature flow (MCF) in Euclidean space, initiated in \cite{EH_entire} and \cite{EH_interior}, was arguably one of the most dramatic early achievements in geometric flows. In particular, the latter paper included the so-called Ecker-Huisken gradient estimate, which was the first instance of ``pseudo-locality'' to be discovered. As every hypersurface is locally a graph, the estimates of \cite{EH_interior} are  useful in the study of the evolution of compact hypersurfaces by MCF as well. Global graphicality, however, is much too stringent an assumption in this setting.  

\bigskip

If a time oriented Lorentzian manifold $(\tilde{\mathcal{N}}^4,g)$ contains a  Cauchy surface $S$ (or equivalently, if $\tilde{\mathcal{N}}$ is globally hyperbolic) then there exists a global time function $t:\tilde{\mathcal{N}}\rightarrow \mathbb{R}$ and a projection map $\pi:\tilde{\mathcal{N}}\rightarrow S$ such that 
\[
(\pi,t):\tilde{\mathcal{N}}\rightarrow S \times \mathbb{R}
\]  
is a diffeomorphism \cite{Geroch, Hawking_Ellis, BS}. Global graphicality is therefore a very natural physical assumption in the Lorentzian setting, even when $S$ is compact. The study of mean curvature flow in this setting, which is called a \textit{cosmological spacetime}, was initiated by Ecker and Huisken \cite{Ecker_Huisken_lor}, where the MCF (and related flows) were used to construct, via parabolic methods, { constant mean curvature} surfaces in such spacetimes, assuming they satisfy the timelike convergence condition, and contain some barriers to the flow.  

\bigskip

While there has been a continued interest in spacelike MCF in the Lorentzian setting (see for instance \cite{Ecker_null, Ecker_mink,Ecker_flat,Smock_evol,Lamb_neu,RS}), the only instance where a local gradient estimate was obtained was for the flat Minkowski space in \cite{Ecker_mink}. 

\bigskip

In the current paper, we are concerned with the implications of being an immortal local in space  mean curvature flow of space-like hypersurfaces in de Sitter space. More concretely, given $R\leq \infty$, consider the manifold 
\begin{equation}
\mathcal{N}_R=\{(x,t)\in \mathbb{R}^3\times \mathbb{R}\;|\; |x| \leq R \}
\end{equation}
endowed with the de Sitter metric 
\begin{equation}\label{ds_metric_intro}
g=e^{2t}(dx_1^2+dx_2^2+dx_3^2)-dt^2.
\end{equation}

\bigskip

We say that a mean curvature flow $(M_s)_{s\in [0,s_0)}$ is graphical in $\mathcal{N}_R$ if the $t$ component of the evolved (spacelike) hypersurfaces $M_s$ can be written as a graph of a function $u:B_R\times [0,s_0)\rightarrow \mathbb{R}$. 

\begin{remark}\label{pseudo_local_nature} 
It will turn out that imposing  a uniform bound on the mean curvature of $M_s$ is natural from the physical standpoint (see Corollary \ref{main_cor} and the discussion surrounding it). We note that contrary to the Riemannian setting, a bound on the mean curvature $H$ does not provide any bound on the time derivative of $u$.  Indeed, as 
\begin{equation}
\frac{d}{ds}u= Hv
\end{equation}
where
\[
v=-g(\partial_t,\nu)\geq 1,
\]
the velocity is huge if $v$ is large, even if $H=1$. 
\end{remark}

\begin{theorem}\label{main_thm}
There exists a universal $R<\infty$ with the following significance: Let  $(M_s)_{s\in [0,\infty)}$   be a graphical mean curvature flow in $\mathcal{N}_R\cap \{t\geq 0\}$ with bounded, non negative mean curvature, and  with graphical function $u(x,s)$. Then
\begin{equation}
{\lim_{\lambda \rightarrow \infty}u(0,\lambda)=\infty,}
\end{equation}
and setting
\begin{equation}
u^{\lambda}(x,s)=u(e^{-u(0,\lambda)}x,s+\lambda)-{{u(0,\lambda)}}
\end{equation}
we have that 
\begin{equation}
u^{\lambda}\xrightarrow{C^{\infty}_{\mathrm{loc}}(\mathbb{R}^3\times \mathbb{R})} 3s,
\end{equation}
as $\lambda\rightarrow \infty$.
\end{theorem}
In light of Remark \ref{pseudo_local_nature}, Theorem \ref{main_thm} is of the ``pseudo local'' nature, even though $H$ is assumed to be bounded.  

\bigskip

To put the theorem within a more geometric perspective, observe that for every $a\in \mathbb{R}$, 
\begin{equation}\label{O_a_def}
O_{a}(x,t):=(e^a x,t-a).
\end{equation}
is an isometry of $\mathcal{N}_{\infty}$. Note further that $(\bar{M}_s)_{s\in (-\infty,\infty)}$ defined by
\begin{equation}\label{barms_def}
\bar{M}_s=t^{-1}(3s)
\end{equation}
is a mean curvature flow in $\mathcal{N}_{\infty}$  satisfying  
\begin{equation}
O_{3a}\left(\bar{M}_s\right)=\bar{M}_{s-a}.
\end{equation}
{Observe further that the} $\bar{M}_s$ form a flat {constant mean curvature umbilical} slicing of {$\mathcal{N}_{\infty}$}.
 
\begin{theorem}[Main theorem - geometric version]\label{main_thm_geom}
There exists a universal $R<\infty$ with the following significance: Let  $(M_s)_{s\in [0,\infty)}$   be a graphical mean curvature flow in $\mathcal{N}_R\cap \{t\geq 0\}$ with bounded, non negative mean curvature, and  with graphical function $u(x,s)$. Then
\begin{equation}
{ \lim_{\lambda \rightarrow \infty}u(0,\lambda)=\infty,}
\end{equation}
and setting $M^{\lambda}_s:=O_{u(0,\lambda)}\left(M_{\lambda+s}\right)$, we have
\begin{equation}
M^{\lambda}_s \xrightarrow{C^{\infty}_{\mathrm{loc}}(\mathbb{R}^3\times \mathbb{R})} \bar{M}_s
\end{equation}
as $\lambda\rightarrow \infty$.
\end{theorem}
Thus, any immortal graphical mean convex mean curvature flow {$(M_s)_{s\in [0,\infty)}$} with bounded curvature in the future of a de Sitter cylinder $\mathcal{N}_R$ { with \textit{finite} $R$,}  converges\footnote{{Note that applying $O_{u(0,\lambda)}$ on $M_{s+\lambda}$ amounts to ``re-defining'' the $(x,t)$ co-ordinates in such a way that (i) shifts the height to where the spatial origin of the flow had reached, (ii) the de-Sitter metric still has the same form under this new $(x,t)$ co-ordinates,  and (iii) the flow time $s$ is shifted to be centered around $\lambda$. Thus,  $M^{\lambda}_s$ truly describes how the flow looks like, centering at flow time $\lambda$.}} to the ``universal'' mean curvature flow $\bar{M}_s$ on the \textit{full de Sitter {hemi-}space} $\mathcal{N}_{\infty}$.\footnote{Our $\mathcal{N}_\infty$ corresponds to points $(t,y_1,\ldots y_4)$ in the hyperboloid model of de Sitter space having $y_1>0$.}

\bigskip

While the proof method of Theorem \ref{main_thm} is similar to the ones in \cite{EH_interior} and \cite{Ecker_mink}, note that the result is of a somewhat different character. Namely, our result states that local graphicality implies \textit{improved graphicality} on \textit{larger and larger} intrinsic balls in $M_s$, which also become flatter and flatter. In other words, contrary to \cite{EH_interior} and \cite{Ecker_mink}, the interior gradient estimate is one that prevails for all time, and improves with time. While this is not how things are proved, intuitively, one can think that the polynomial deterioration of the bounds in  \cite{EH_interior} and \cite{Ecker_mink} as time progresses are absorbed by the exponential  expansion of the de Sitter metric \eqref{ds_metric_intro}. This indicates some interplay between the wave equation nature of the ambient metric{, the cosmological constant,}  and the heat type equation of the evolved hypersurface.

\bigskip

On a technical level, the differences in the proof of Theorem \ref{main_thm} in comparison with the ones in \cite{EH_interior,Ecker_mink} include (I) a barrier argument, which forces flows as in Theorem \ref{main_thm} to go to larger and larger $t$ (see Proposition \ref{barrier_proposition} ), (II) a flow time independent (but $t$ dependent) cutoff in the gradient estimate (see \eqref{r_def} and Theorem \ref{grad_est}), and (III) the use of the $v$ equation directly, rather than equation for $A$ or its traceless part, in the flatness estimate (See Theorem \ref{flat_est}).

\subsection*{Relation to the cosmic no hair conjecture}\label{cnh_relate}

Let us change the perspective a little bit, in order to discuss how the above result fits into the mean curvature flow program for addressing the cosmic no hair conjecture \cite{KS, CSV, CHSV}. 

Suppose that the $4$ dimensional Lorentzian manifold $(\tilde{\mathcal{N}},\tilde{g})$ forms a cosmological spacetime, that is, there exist a global time function $t:\tilde{\mathcal{N}}\rightarrow \mathbb{R}$ such that $\nabla t \neq 0$ is timelike and such that the $t$ level set hypersurfaces $\{t=t_0\}$ are compact. Assume further that (A) the metric $\tilde{g}$ satisfies the non-vacuum Einstein equation with cosmological constant $3$,
\begin{equation}
\mathrm{Ric}_{\tilde{g}}-\frac{R_{\tilde{g}}}{2}\tilde{g}=-3\tilde{g}+\tilde{T},
\end{equation}
where the stress energy tensor $\tilde{T}$ satisfies the dominant energy condition and the strong energy condition\footnote{Namely, $\tilde{T}(X,X)-\frac{1}{2}\mathrm{tr}\tilde{T}\tilde{g}(X,X)\geq 0$ and $-(\iota_X \tilde{T})^{\#}$ is future directed or null whenever $X$ is future directed }, that (B) the hypersurface $\{t=0\}$ is mean convex w.r.t to the future pointing normal, and that (C) there exist an open set $\Omega \subseteq \{t>0\}$  and a constant $C<\infty$ such that $\partial \Omega$ is smooth, spacelike, and has non positive mean curvature w.r.t the normal pointing into $\Omega$, and such that 
\[
-\tilde{g}(\nabla t,\nabla t) \leq C
\]
on $\{t\geq 0\}-\Omega$.

The following is an immediate application of Theorem \ref{main_thm}, and \cite[Theorem 2, Theorem 3]{CSV} (which itself relies one  \cite[Theorem 4.1]{Ecker_Huisken_lor}).
\begin{corollary}\label{main_cor}
There exists a universal $R<\infty$ with the following significance. Suppose that $\tilde{\mathcal{N}}$ satisfies (A),(B) and (C) as above, and suppose that  there exists a $t$ preserving Lorentzian embedding $\iota: \mathcal{N}_R\rightarrow \tilde{\mathcal{N}}$. Then the graphical mean curvature flow $M_s$ emanating from $\{t=0\} \subseteq \tilde{\mathcal{N}}$ exists for all time, is mean convex and has bounded mean curvature. In particular $\iota^{-1}(M_s)$ satisfies the conclusion of Theorem \ref{main_thm}.  
\end{corollary}

The above corollary is an (easy) test case for the mean curvature flow program for addressing the cosmic no hair conjecture, as we now briefly explain.  The cosmic no hair conjecture  is an informal physics conjecture attributed to Gibbons and Hawking \cite{GibonHawking} and to Hawking and Moss \cite{HawkingMoss}  stating that every initially expending universe with positive cosmological constant should asymptote, away from black holes, to de Sitter space. Assumptions (A),(B) and (C) above form a formalization of such a universe\footnote{See \cite{CHSV} for a discussion regarding the physical justification of these assumptions.}.  In the physics paper \cite{KS} by Kleban and the second named author, it was suggested that mean curvature flow emanating from $t^{-1}(0)$  might serve as a useful tool for probing such ambient manifold $\tilde{\mathcal{N}}$, as it both avoids the potential singular region $\Omega$ and as the area of the spacelike surfaces is non decreasing. In the $(2+1)$ dimensional setting  \cite{CSV} and under surface symmetry assumptions in the $(3+1)$ dimensional case \cite{CHSV}, it was further shown that the mean curvature flow itself provides a slicing of spacetime which implies various metric convergence results to de Sitter space and its standard slicing $\bar{M}_s$ above -providing both formalizations and proofs of the conjecture in these simplified cases. If one can show that such convergence will hold somewhere along the MCF evolution of $t^{-1}(0)$ in $(3+1)$ dimensions \textit{without symmetry assumption}, this will constitute a proof of the no hair conjecture. In particular, as is explained in \cite{CHSV}, this will also provide  further and important theoretical justification for the physical theory of cosmic inflation. 

\bigskip

While there are some global mechanism which might play a role in showing that the mean curvature flow of $t^{-1}(0)$ converges on some expanding region to the standard slicing of de Sitter space, it is reasonable to expect that some local mechanisms will play a role too. Corollary \ref{main_cor} indicates such a mechanism for local (expanding) in space and global in time convergence to the flat slicing, \textit{assuming} the space is locally de Sitter. This is of course very far from what is needed for the cosmic no hair conjecture, but it indicates the { potential usefulness} of such ideas in the general context.
\subsection*{Acknowledgements}
We would like to thank Gerhard Huisken and Ben Lambert for inspiring discussion. Part of this work was carried out during the second author's visit to the Hebrew University in the context of the program surrounding Prof. Huisken distinguished Gordon visiting Professorship at the Hebrew University. OH was partially supported by ISF grant 437/20. LS is partially supported by the SNSF grant $200021\_213120$.

\section{Preliminaries and notation}
Recall that for every $R\leq \infty$, we endow the manifold
\begin{equation}
\mathcal{N}_R=\{(x,t)\in \mathbb{R}^3\times \mathbb{R}\;|\; |x| \leq R \}
\end{equation}
with the de Sitter metric 
\begin{equation}\label{ds_metric}
g=e^{2t}(dx_1^2+dx_2^2+dx_3^2)-dt^2.
\end{equation}
The curvature tensor of de Sitter space is given by
\begin{equation}\label{ds_curvature}
R(X,Y,Z,W)=g(X,W)g(Y,Z)-g(X,Z)g(Y,W)
\end{equation}
We will denote by $\bar{\nabla}$ the covariant derivative in $\mathcal{N}_R$, and by $\p_{i}$ and $\p_t$ the coordinate vector fields.
\begin{lemma}[Covariant derivatives]
We have
\begin{equation}\label{nabla_t_eq}
\bar{\nabla}_X \partial_t= X+g(X,\p_t)\p_t, \qquad \bar{\nabla}_{\partial_i}\partial_i=e^{2t}\partial_t,\qquad \bar{\nabla}_{\partial_t}\partial_i=\bar{\nabla}_{\partial_i}\partial_t=\partial_{i}.
\end{equation}
\end{lemma}
\begin{proof}
For the first equation, note that as $g(\partial_t,\partial_t)=-1$, for every vector field $X$ we have 
\[
g(\bar{\nabla}_X\partial_t,\partial_t)=0.
\]
Moreover $\bar{\nabla}_{\p_t}\p_t=0$ since for every $p$, $t\mapsto(p,t)$ is a geodesic.  Now, if $Y$ is a vector field which is tangent to the slices $\{t=c\}$, and which is $t$ independent, \eqref{ds_metric} implies
\begin{equation}
g(\bar{\nabla}_Y\p_t,Y)=g(\bar{\nabla}_{\p_t}Y,Y)=\frac{1}{2}\partial_t g(Y,Y)=g(Y,Y).
\end{equation}
By polarization, it follows that $X\mapsto \bar{\nabla}_X\partial_t$ is the projection map to the $\{t=c\}$ slice, which is precisely the first equation. The third equation follows from the first equation and symmetry. The second equation follows from Koszul's formula.
\end{proof}
We can also directly compute { the gradients of the special coordinate functions to be}
\begin{equation}\label{cor_grads}
\bar{\nabla} x_i=e^{-2t}\partial_i, \qquad \bar{\nabla} t=-\partial_t.
\end{equation}
We can also compute the Hessian and the wave operator of the coordinate functions. Recall that for a function $f$ the Hessian is defined, just like in the Riemannian case, as
\[
\overline{\mathrm{Hess}}f(X,Y)=XYf-\bar{\nabla}_XYf
\]
and the wave operator is given by
\[
\bar{\Box} f=g^{kl}\overline{\mathrm{Hess}}f(\p_k,\p_l).
\] 
Therefore, in our de Sitter setting, in light of \eqref{nabla_t_eq},
\begin{equation}\label{cor_hess1}
\overline{\mathrm{Hess}}x_i(\partial_{i},\partial_t)=-1,
\end{equation}
and all the other entries are zero, and 
\begin{equation}\label{cor_hess2}
\overline{\mathrm{Hess}}t(\partial_i,\partial_i)=-e^{2t},
\end{equation}
and all the other entries are zero. In particular
\begin{equation}\label{co-ordinate_wave}
\bar{\Box} x_i=0, \qquad \bar{\Box} t=-3.
\end{equation}

\subsection*{Spacelike graphical hypersurfaces}
Consider a spacelike hypersurface $M$, graphical over the $x_1x_2x_3$ plain, and let $\nu$ be its future pointing unit normal, i.e. the unit normal such that $g(\p_t,\nu)<0$, and consider the \textit{gradient function} $v$, given by  
\begin{equation}
v=-g( \p_t,\nu).
\end{equation}
Observe that $v\geq 1$. Given a point $p\in M$, take a normal frame $\{e_1,e_2,e_3\}$ around $p$, which we assume to be the principal directions. We let $A$ be the scalar second fundamental form, whose components are given by 
\[
h_{ij}=g(\bar{\nabla}_{e_i}\nu,e_j)=-g(\bar{\nabla}_{e_i}e_j,\nu).
\] 
In particular $\bar{\nabla}_{e_i}e_j=h_{ij}\nu$. The mean curvature is then given by
\[
H=h_{ii}.
\]
In light if \eqref{ds_curvature}, the Codazzi equation takes the form 
\[
\bar{\nabla}_{k}h_{ij}=\bar{\nabla}_{i}h_{jk}.
\]

\bigskip

Coming back to the analytic properties of the coordinate functions when restricted to $M$, if we denote by $\nabla$ the gradient operator on $M$, \eqref{cor_grads} implies
\begin{equation}
\nabla x_i=e^{-2t}\partial_i+e^{-2t}g(\partial_i,\nu)\nu,\qquad\nabla t=-\p_t+v\nu.
\end{equation}
In particular
\begin{equation}\label{grad_sq}
|\nabla x_i|^2=e^{-2t}+e^{-4t}g(\nu ,\partial_i)^2,\qquad |\nabla t|^2=v^2-1,
\end{equation}
and 
\begin{equation}\label{cross_term}
g(\nabla x_i, \nabla t)= e^{-2t}vg(\p_i,\nu).
\end{equation}

Finally, observe that for a spacelike hypersurface $M$ and for a smooth function $f:\mathcal{N}_R\rightarrow \mathbb{R}$  we can relate the Laplacian of $f$ w.r.t $M$, $\Delta f$ to the wave operator on $\mathcal{N}_R$ by the formula
\begin{equation}\label{wave_lap}
\Delta f= \bar{\Box} f +H\nu f+ \overline{\mathrm{Hess}}f(\nu,\nu).
\end{equation}
Combining  \eqref{wave_lap}, \eqref{co-ordinate_wave} \eqref{cor_grads} and  \eqref{cor_hess1} we see that
\begin{equation}\label{delta_xieq}
\Delta x_i=He^{-2t}g(\nu,\partial_i)-2e^{-2t}vg(\nu,\partial_{i}),
\end{equation}
and using  \eqref{cor_hess2} we similarly get
\begin{equation}\label{delta_teq}
\Delta t= -3+Hv-(v^2-1).
\end{equation}

\section{evolution equations}
\subsection*{Gradient evolution} Following \cite{Bartnik} and \cite{Ecker_Huisken_lor}, we start by computing the evolution equation for $v$. Contrary to the general Lorentzian situation, but similarly to the case of Minkowski space \cite{Ecker_mink}, the evolution is good enough to obtain ``pseudo local'' estimates for $v$.  

\begin{lemma}[Evolution of $v^2$]
\begin{equation}\label{vsq_eq}
\left(\partial_s-\Delta\right)v^2 =4Hv-2v^4-4v^2-2|A|^2v^2+2A(e_i,\p_t^{\top})^2-4|\nabla v|^2.
\end{equation}
\end{lemma}
\begin{proof}
Using \eqref{nabla_t_eq}, we compute
\begin{equation}\label{grad_eq}
\nabla_{i}v=g(e_i,\p_t)v-g(\partial_t,h_{ij}e_j), 
\end{equation}
from which we infer
\begin{equation}\label{nablavsq}
|\nabla v|^2  = v^2(v^2-1)-2A\left(\p_t^{\top},\p_t^{\top}\right)v+A\left(e_i,\partial_t^{\top}\right)^2.
\end{equation}
Differentiating \eqref{grad_eq} once more, and using \eqref{nabla_t_eq} and the Codazzi equation,  we obtain
\begin{align}\label{deltav_eq}
\Delta v&=-Hv^2\nonumber+(3+g(e_i,\p_t)^2)v+ g(e_i,\p_t)^2v \nonumber \\
&-2g(e_i,\p_t)g(e_j,\p_t)h_{ij}-H-g(\p_t,\nabla H)+|A|^2v\\
&=-H(v^2+1)+2v^3+v\nonumber -g(\partial_t,\nabla H)+|A|^2v -2 A(\partial_t^{\top},\p_t^{\top})\\
&= -H(v^2+1)+v^3+2v\nonumber -g(\partial_t,\nabla H)+|A|^2v-\frac{A\left(e_i,\partial_t^{\top}\right)^2}{v}+\frac{|\nabla v|^2}{v},
\end{align}
where in the last equation we have used \eqref{nablavsq} again. 
Using \cite[Proposition 3.1]{Ecker_Huisken_lor} and \eqref{nabla_t_eq} we have
\begin{equation}\label{vtime_der}
\partial_{s}v=-g( \p_t, \nabla H)-Hg\left(\bar{\nabla}_{\nu}\partial_t,\nu \right)=-g( \p_t, \nabla H)-H(v^2-1),
\end{equation}
which, put together with \eqref{deltav_eq}, yields
\begin{equation}
\left(\partial_s-\Delta\right)v =2H-v^3-2v-|A|^2v+\frac{A(e_i,\p_t^{\top})^2}{v}-\frac{|\nabla v|^2}{v},
\end{equation}
from which the asserted result follows readily. 
\end{proof}

We now derive two differential inequality for $v^2$. The first will be useful to obtain the interior gradient estimate, and the second will be useful for the flatness estimate.  
\begin{corollary}
For every $\delta\in [0,1/3]$ 
\begin{equation}\label{v_evolve_ineq}
(\partial_s-\Delta) v^2 \leq -(4+\delta)|\nabla v|^2 -2(1-\delta) v^4 +2H^2v^2+4Hv.
\end{equation}
\end{corollary}
\begin{proof}
We argue similarly to \cite[Theorem 3.1]{Bartnik} and \cite[Proposition 4.4]{Ecker_Huisken_lor}. Setting $\lambda_1=\max_{|w|=1}{|A(w,w)|}$, we can clearly bound
\begin{equation}\label{lest_vsq_term}
A(e_i,\partial_t^{\top})^2=g(e_i,\partial_t)^2A(e_i,e_i)^2 \leq \lambda_1^2 g(e_i,\partial_t)^2=\lambda_1^2(v^2-1).
\end{equation}
On the other hand, it follows from \cite[Eq. 3.17]{Bartnik} that for any spacelike hypersurface we have the bound 
\begin{equation}
|A|^2 \geq \frac{4}{3}\lambda_1^2-H^2.
\end{equation}
Plugging the above two inequalities into \eqref{vsq_eq} yields 
\begin{equation}\label{vsq_ineq1}
\left(\partial_s-\Delta\right)v^2 \leq 4Hv+2H^2v^2-2v^4-\frac{2}{3}\lambda_1^2v^2-4|\nabla v|^2.
\end{equation}
Finally, using  \eqref{nablavsq}, \eqref{lest_vsq_term} and using the absorbing inequality gives
\begin{equation}\label{gradv_norm_2}
|\nabla v|^2 \leq v^2(v^2-1)+2\lambda_1v(v^2-1)+\lambda_1^2(v^2-1) \leq 2v^2(v^2-1)+2\lambda_1^2(v^2-1).
\end{equation}
Combining this with \eqref{vsq_ineq1} yields the desired result.
\end{proof}

\begin{corollary}[Second $v$ inequality]
We can estimate
\begin{equation}\label{v_eq_for_ambilic}
(\partial_s-\Delta) v^2 \leq -4|\nabla v|^2-2(v^2-1).
\end{equation}
\end{corollary}
\begin{proof}
\eqref{lest_vsq_term} implies 
\begin{equation}
A(e_i,\partial_t^{\top})^2\leq |A|^2(v^2-1), 
\end{equation}
while clearly 
\begin{equation}
4Hv= 2\frac{H^2}{3}+6v^2-2\left(\frac{H}{\sqrt{3}}-\sqrt{3}v\right)^2.
\end{equation}
Substituting this into \eqref{vsq_eq} we obtain
\begin{align}\label{vsq_ineq3}
\left(\partial_s-\Delta\right)v^2 &\leq -2v^2(v^2-1)-4|\nabla v|^2-2\left(\frac{H}{\sqrt{3}}-\sqrt{3}v\right)^2-2\left(|A|^2-\frac{H^2}{3}\right)\\
&\leq -2(v^2-1)-4|\nabla v|^2,
\end{align}
where for the second inequality we have identified $|A|^2-H^2/3$ as the norm of the traceless second fundamental form. 
\end{proof}

\begin{remark}
Note that the equation \eqref{v_eq_for_ambilic} by itself would have sufficed to conclude flatness in the close setting, i.e. where \eqref{ds_metric} was supported on the torus $\mathbb{T}^3\times \mathbb{R}$. Getting a gradient term with coefficient more than $4$ as in \eqref{v_evolve_ineq} is what is needed for an unconditional interior gradient estimate.
\end{remark}

\subsection*{Cut-off evolution} For $\alpha>0$, let us consider the function $r=r_{\alpha}$ given by  
\begin{equation}\label{r_def}
r(x_1,x_2,x_3,t)=r_{\alpha}(x_1,x_2,x_3,t):=e^{\alpha t}(x_1^2+x_2^2+x_3^2)=e^{\alpha t}|x|^2.
\end{equation}
$r_{\alpha}$  will serve as a building block for the cutoff functions involved in the local estimates. We now compute its evolution.

\begin{lemma}\label{r_evolve_ineq}
For every $\varepsilon>0$ and $\alpha\in (0,2)$ there exists $T=T(\alpha,\varepsilon)$ such that if $t\geq T$ we have 
\begin{equation}
(\partial_s-\Delta)r \geq (-\alpha^2 r-\varepsilon)v^2
\end{equation}
\end{lemma}
\begin{proof}
Using \eqref{cross_term}, we get 
\begin{equation}\label{grad_interact}
g(\nabla(x_i^2),\nabla(e^{\alpha t}))=2\alpha x_ie^{(\alpha-2)t}vg(\p_i,\nu).
\end{equation}
Using \eqref{delta_xieq} and \eqref{grad_sq}, we get
\begin{equation}\label{lapsq}
\Delta x_i^2=2x_ie^{-2t}\left(H-2v)g(\nu,\partial_i\right)+2e^{-2t}+2e^{-4t}g(\partial_i,\nu)^2.
\end{equation}
Similarly, using \eqref{delta_teq} and \eqref{grad_sq}, we get
\begin{align}\label{lapealphat}
\Delta e^{\alpha t}=&-3\alpha e^{\alpha t}+\alpha Hve^{\alpha t}-\alpha e^{\alpha t}(v^2-1)+\alpha^2 e^{\alpha t}(v^2-1).
\end{align}
Combining \eqref{grad_interact},  \eqref{lapsq} and \eqref{lapealphat}  we get
\begin{align}\label{Deltar_eq}
\Delta r=& -3\alpha r+\alpha Hvr-\alpha r(v^2-1)+\alpha^2 r(v^2-1) \nonumber\\
&+2x_ie^{(\alpha-2)t}\left(H-2v)g(\nu,\partial_i\right)+6e^{(\alpha-2)t}+2e^{(\alpha-2)t}(v^2-1)\\
&+4\alpha x_ie^{(\alpha-2)t}vg(\p_i,\nu) \nonumber.
\end{align}
{
Now, given $\delta>0$ we can bound
\begin{equation}\label{cross_term2}
|x_ie^{(\alpha-2)t}vg(\nu,\partial_i)| \leq \delta|x_i|^2e^{\alpha t}v^2+\frac{1}{\delta}e^{(\alpha-4)t}g(\nu,\partial_i)^2\leq \delta r v^2+\frac{1}{\delta}e^{(\alpha-2)t}v^2.
\end{equation}
Note that if $\delta<\alpha$ then 
\begin{equation}\label{no_eps_r}
\delta r v^2 \leq 3\alpha r+\alpha r(v^2-1).
\end{equation}
}
Since $\alpha\in (0,2)$, making use of all the negative exponents, we see from \eqref{cross_term2} \eqref{no_eps_r} and \eqref{Deltar_eq} that
\begin{equation}\label{delta_rbd}
\Delta r \leq \alpha Hvr +2x_ie^{(\alpha-2)t}Hg(\nu,\partial_i)+\alpha^2 rv^2+\varepsilon v^2,
\end{equation}
if $t\geq T$ for $T$ large enough. 

On the other hand, in light of \eqref{cor_grads}, we can compute the time evolution of $r$ as 
\begin{equation}
\partial_s r= \alpha Hvr+2x_ie^{(\alpha-2)t}H g(\nu,\partial_i).
\end{equation}
Combining this with \eqref{delta_rbd} gives the desired result.
\end{proof}

We will also need a two sided bound for the gradient of $r$. 
\begin{lemma}\label{grad_r_size}
For every $\varepsilon>0$ and $\alpha\in (0,2)$ there exists $T=T(\alpha,\varepsilon)$ such that if $t\geq T$
\begin{equation}
 \alpha^2(1-\varepsilon)r^2(v^2-1)-\varepsilon rv^2 \leq |\nabla r|^2 \leq 2\alpha^2r^2(v^2-1)+\varepsilon r v^2
\end{equation}
\end{lemma}
\begin{proof}
Since 
\[
\nabla r= \alpha r\nabla t +2e^{\alpha t}x_i\nabla x_i.
\] 
then by virtue of \eqref{grad_sq}, we have
\begin{equation}
|\nabla r|^2 \geq (1-\varepsilon)\alpha^2 r^2|\nabla t|^2-\frac{4x_i^2e^{2\alpha t}}{\varepsilon}|\nabla x_i|^2\geq \alpha^2(1-\varepsilon)r^2(v^2-1)-\varepsilon rv^2,
\end{equation}
provided $T$ is large enough, which gives the lower bound. The upper bound similarly follows.
\end{proof}

\subsection*{Curvature evolution} While it is not strictly necessary for our current discussion, we record here the evolution equation of curvatures.
\begin{lemma}
The second fundamental form satisfies the evolution
\begin{align}
\left(\frac{d}{ds}-\Delta\right)h_{ij}=&2Hh_{ik}h_{jk}-2H\delta_{ij}-h_{ij}(|A|^2-9)
\end{align}
and
\begin{equation}
\left(\frac{d}{ds}-\Delta\right)|A|^2=-2|\nabla A|^2-4H^2+2|A|^2(9-|A|^2) 
\end{equation}
In particular, 
\begin{equation}
\left(\frac{d}{ds}-\Delta\right)\left(|A|^2-H^2/3\right) \leq 18\left(|A|^2-H^2/3\right)-\frac{2H^2}{3}\left(|A|^2-H^2/3\right).
\end{equation}
\end{lemma}
\begin{proof}
Using \cite[Prop. 3.3]{Ecker_Huisken_lor} and $\Ric(\nu,\nu)=-3$ we obtain.
\begin{align}
\left(\frac{d}{ds}-\Delta\right)h_{ij}=&2Hh_{ik}h_{jk}-h_{ij}(|A|^2-3)-2h_{kl}(\delta_{lk}\delta_{ij}-\delta_{lj}\delta_{ik})\\
&-h_{jl}(\delta_{lk}\delta_{ki}-\delta_{li}\delta_{kk})-h_{li}(\delta_{lk}\delta_{kj}-\delta_{lj}\delta_{kk})\\
&=2Hh_{ik}h_{jk}-h_{ij}(|A|^2-3)-2H\delta_{ij}+2h_{ij}-h_{ij}+3h_{ij}-h_{ij}+3h_{ij},
\end{align}
from which the first equation follows. The second equation follows from the first one after recalling that
\begin{equation}
\frac{d}{ds}g^{ij}=-2Hh_{ij}.
\end{equation}
Finally
\begin{align}
\frac{d}{ds}\left(|A|^2-\frac{H^2}{3}\right)&=-2|\nabla A|^2-4H^2+2|A|^2(9-|A|^2)+\frac{2}{3}|\nabla H|^2-\frac{2}{3}{H^2}(3-|A|^2)\\
&=2(9-|A|^2)\left(|A|^2-\frac{H^2}{3}\right)-2|\nabla A|^2+\frac{2}{3}|\nabla H|^2\\
&=2(9-|A|^2)\left(|A|^2-\frac{H^2}{3}\right)-2\left|\nabla \left(A-\frac{H}{3}g\right)\right|^2,
\end{align}
From which the last equation follows by adding and subtracting $\frac{2H^2}{3}\left(|A|^2-\frac{H^2}{3}\right)$.
\end{proof}

\section{Gradient estimates}

We are now ready to derive the local in space but global in time gradient estimate. Let us set
\begin{equation}
D_{\alpha,R}:=\{(x,t)\;|\; e^{\alpha t}|x|^2\leq R\}.
\end{equation}
Observe that for every $\alpha\in (0,2)$ the diameter of $D_{\alpha,R}\cap \{t=t_1\}$ goes to infinity as $t_1\rightarrow \infty$. 

\begin{theorem}[Gradient estimate]\label{grad_est}
There exists some $\alpha_0<2$ such that for every $\alpha<\alpha_0$ and $\Lambda<\infty$, there exist $T_0=T_0(\alpha)<\infty$, $R_0=R_0(\alpha){\in (2,\infty)}$ and  $C=C(\Lambda)<\infty$ with the following significance: Suppose $R\geq R_0$ and suppose that $M_s$ is a mean convex mean curvature flow of graphs in $\mathcal{N}_R$ such that $M_s$ is contained in $t\geq T_0$ and such that $\sup_s\sup_{M_s}H \leq \Lambda$. {
 Then there exists $s_0=s_0(M_0,\Lambda)<\infty$ such that for every $s\geq s_0$ we have 
\begin{equation}
\sup_{z\in M_s \cap D_{\alpha,R/2}} v \leq C.
\end{equation}
}
\end{theorem}
\begin{proof}
Let $\alpha\in (0,2)$, to be fixed later, and let $r=r_{\alpha}$. Take $p>1$, to be fixed later, and consider the function $\mu(r)=(R-r)^p$. Note that whenever $r<R$ we have
\begin{equation}
\mu'\leq 0,\qquad \mu'(r)(R-r)=-p\mu,\qquad \mu''(r)(R-r)^2=p(p-1)\mu.
\end{equation}
Given $\varepsilon>0$, applying Lemma \ref{r_evolve_ineq} we see that 
\begin{equation}
(\partial_s-\Delta)\mu(r) \leq  (\varepsilon+\alpha^2 r)\frac{p}{R-r}\mu(r)v^2 -\frac{p(p-1)}{(R-r)^2}\mu(r)|\nabla r|^2,
\end{equation}
provided $T_0(\varepsilon)$ is large enough. Combining this with \eqref{v_evolve_ineq} we get
\begin{align}\label{first_evol_local}
(\partial_s-\Delta)&(v^2\mu(r)) \leq -(4+\delta)\mu(r)|\nabla v|^2-2(1-\delta)\mu(r)v^4+\mu (2\Lambda^2v^2+4\Lambda v) \nonumber \\
 &+(\varepsilon+\alpha^2 r)\frac{p}{R-r}\mu(r)v^4-\frac{p(p-1)}{(R-r)^2}\mu v^2|\nabla r|^2+4v\mu\frac{p}{R-r} \nabla v \cdot \nabla r. 
\end{align}
Using the absorbing inequality, we can estimate
\begin{equation}
\left|4v\frac{p}{R-r} \nabla v\cdot \nabla r\right| \leq (4+\delta)|\nabla v|^2+ \frac{4p^2}{(4+\delta)(R-r)^2}|\nabla r|^2v^2,
\end{equation}
so plugging this into \eqref{first_evol_local} we obtain
\begin{align}\label{second_evol_local}
(\partial_s-\Delta)(v^2\mu(r)) \leq v^2\mu\Big(&-2(1-\delta)v^2+ (2\Lambda^2+4\Lambda)+\frac{p(\varepsilon+\alpha^2 r)}{R-r}v^2 \nonumber\\
&+\left(\frac{4p^2}{4+\delta}-p(p-1)\right)\frac{|\nabla r|^2}{(R-r)^2}\Big).
\end{align}
{Setting} $\delta=1/3$, let us {choose} $p$ sufficiently large so that we have 
\begin{equation}\label{p_size}
\frac{4p^2}{4+\delta}-p(p-1) \leq -\frac{p^2}{20},
\end{equation}
and assume $\varepsilon$ is such that { $p\varepsilon<\frac{1}{4}$} (we will later further restrict  $\varepsilon$). Substituting this choice of parameters into  \eqref{second_evol_local}, and using Lemma \ref{grad_r_size} yields  
\begin{align}\label{second_evol_local2}
(\partial_s-\Delta)(v^2\mu(r)) \leq v^2\mu\Big(&-\frac{4}{3}v^2+ (2\Lambda^2+4\Lambda)+\frac{1+4\alpha^2pr}{4(R-r)}v^2\nonumber \\
&-\frac{p^2}{20(R-r)^2}\left((1-\varepsilon)\alpha^2r^2(v^2-1)-\varepsilon rv^2\right) \Big).
\end{align}

As we are eventually interested in a bound for $v$, from this point onward we are going to restrict to the case that $v$ is large. Indeed, if $v\geq v_0(\Lambda)$ for $v_0\geq 2$ large enough, we have
\begin{equation}\label{curv_cont_bound}
(2\Lambda^2+4\Lambda) <\frac{v^2}{6}, \qquad  (1-\varepsilon)(v^2-1) \geq \frac{4}{5}v^2,
\end{equation}
which, plugged into \eqref{second_evol_local2} yields
\begin{align}\label{third_evol_local}
(\partial_s-\Delta)(v^2\mu(r)) \leq v^4\mu\Big(-\frac{7}{6}+\frac{1+4\alpha^2pr}{4(R-r)}-\frac{p^2}{20(R-r)^2}\left(\frac{4}{5}\alpha^2r^2-\varepsilon r\right) \Big).
\end{align}
Assuming $R\geq 2$ and taking $\varepsilon=\varepsilon(\alpha)$ { sufficiently small} we can estimate
\begin{equation}\label{err_cont_bound}
\frac{p^2\varepsilon r }{20(R-r)^2} \leq \frac{r^2\alpha^2}{100(R-r)^2}+{\frac{1}{12}}.
\end{equation} 
Taking $\alpha_0=\frac{1}{10p}$ and $\alpha<\alpha_0$ (and corresponding $\varepsilon(\alpha)$ and corresponding $T_0(\varepsilon)=T_0(\alpha)$) and plugging \eqref{err_cont_bound} into \eqref{third_evol_local} we finally obtain that whenever $v\geq v_0$, we have the differential, inequality  
\begin{equation}\label{basic_ineq}
(\partial_s-\Delta)(v^2\mu(r)) < v^4\mu\left({-\frac{1}{12}}-1+\frac{1}{4(R-r)}+\frac{\alpha r}{10(R-r)}-\frac{\alpha^2r^2}{40(R-r)^2}\right).
\end{equation}

Our aim is to show that there exists some $R_0(\alpha)\geq 2$ such that \eqref{basic_ineq} can not hold, provided $R\geq  R_0(\alpha)$. To that end, we will analyze  \eqref{basic_ineq} in two regimes: when $r\in [0,R-1]$ and when $r\in (R-1,R)$.   First,  if $r\in [0,R-1]$ then  $\frac{1}{4(R-r)}\leq \frac{1}{4}$ and so \eqref{basic_ineq} implies
\begin{equation}\label{first_case_max}
(\partial_s-\Delta)(v^2\mu(r)) < {-\frac{v^4\mu}{12}} +v^4\mu\left(-\frac{3}{4}+\frac{\alpha\sqrt{3} r}{\sqrt{40}(R-r)}-\frac{\alpha^2r^2}{40(R-r)^2}\right)\leq {-\frac{v^2\mu}{12}}.
\end{equation}
{Alternatively, if} $r\in (R-1,R)$, note that if $R\geq \frac{5}{2\alpha}+1=R_0(\alpha)$ then 
\[
\frac{1}{4(R-r)} \leq \frac{\alpha r}{10(R-r)}
\]
and so \eqref{basic_ineq} implies  
\begin{equation}
(\partial_s-\Delta)(v^2\mu(r)) < {-\frac{v^4\mu}{12}}+ v^4\mu\left(-1+\frac{2\alpha r}{10 (R-r)}-\frac{\alpha^2r^2}{{100}(R-r)^2}\right)<{-\frac{v^2\mu}{12}}.
\end{equation}

We have thus shown that if ${s \geq 0}$ and if {$z\in M_{s}$} is a point with $v\geq v_0$, and if $\alpha<\alpha_0$, $R\geq R_0(\alpha)$,  $t_0 \geq T_0(\alpha)$  and $r<R$ then
\begin{equation}\label{derive}
(\partial_s-\Delta)(v^2\mu(r))(z,s) < { -\frac{\mu v^2}{12}}.
\end{equation}
Suppose that  $s\geq 0$, and that $z_{\ast}$ maximizes $\mu v^2(\cdot,s)$ over the domain
\begin{equation}
D_{\alpha,R}:=\{(x,t)\;|\; e^{\alpha t}|x|^2\leq R\},
\end{equation}
and assume  
\[
\mu v^2(z_{\ast},s) \geq R^pv_0^2.
\]
Since $\mu\leq R^p$, the above condition implies  $v(z_{\ast},s)\geq v_0$, so by spatial maximality and \eqref{derive} we have
\begin{equation}\label{diffeqsome}
\frac{d}{ds}\left(\mu v^2(z_{\ast},s)\right) \leq -\frac{\mu v^2{(z_{\ast},s)}}{12} \leq -\frac{R^pv_0^2}{12}. 
\end{equation}
{Denoting  
\begin{equation}
h(s):=\max_{z \in D_{\alpha,R}} \mu v^2(z,s), 
\end{equation}
and integrating \eqref{diffeqsome} for as long as $h(s)\geq R^pv_0^2$,} we get that for every $s\geq 0$ we have
\begin{equation}
h(s) \leq \max\left\{ R^pv_0^2,h(0)-\frac{R^pv_0^2}{12}s\right\}. 
\end{equation}
Therefore, for $s\geq s_0(M_0,\Lambda) $ we must have 
\begin{equation}
\sup_{z\in D_{\alpha,R}\cap M_s}\mu v^2 \leq R^p v_0^2.
\end{equation}  
In particular,
\begin{equation}
v \leq 2^{p/2}v_0
\end{equation} 
on $D_{\alpha,R/2}$ for $s\geq s_0$. 

\end{proof}

\bigskip

As was mentioned in the introduction, as opposed to the Euclidean and Minkowski case, our flatness estimate is also in the form of a gradient estimate.

\begin{theorem}[Flatness estimate]\label{flat_est}
For every $C<\infty$ and { $\theta\in (0,1)$}  there exist  $T_1=T_1(C,\theta)<\infty$, $s_1=s_1(C,\theta)<\infty$, and $\alpha=\alpha(C,\theta){\in (0,1)}$ with the following significance. Suppose that $(M_s)_{s\in (0,\infty)}$ is a mean convex mean curvature flow of graphs in $D_{\alpha,1}$ such that $M_s$ is contained in $t\geq T_1$ and  such that  
\begin{equation}\label{vA_bd}
v \leq C.
\end{equation}
Then for every every $s\geq s_1$ 
\begin{equation}
\sup_{z\in M_s\cap D_{\alpha,\theta}} v(z) \leq 1+\theta. 
\end{equation}
\end{theorem}

\begin{proof}
Let $\alpha\in (0,1)$, to be fixed later, and let $r=r_{\alpha}$. Consider the function $\phi(r)=(1-r^2)^2$. Letting $\eta>0$ { to be determined later}, we note that  Lemma \ref{r_evolve_ineq} and Lemma \ref{grad_r_size} combined with \eqref{vA_bd} imply that if $\alpha$ is small enough and if $T_1$ is large enough
\begin{equation}\label{small_heat}
(\partial_s-\Delta)\phi(r) \leq  \eta.
\end{equation}
{
Setting 
\[
h=(v^2-1)\phi,
\]
combining \eqref{small_heat}  with \eqref{v_eq_for_ambilic} }we get
\begin{align}\label{first_evol_local2}
(\partial_s-\Delta)h &\leq -4\phi(r)|\nabla v|^2-2\phi(r) (v^2-1)+C^2\eta-4v\phi' \nabla v\cdot \nabla r^2 \nonumber\\
&\leq -2\phi(r) (v^2-1)+C^2\eta+\frac{4v^2(\phi')^2}{\phi}r^2|\nabla r|^2\\
& \leq -2h+10C^2 \eta\nonumber, 
\end{align}
where in the last inequality we have used the upper bound in Lemma \ref{grad_r_size}, the gradient bound \eqref{vA_bd} and have taken $\alpha=\alpha(\eta)$  sufficiently small. Taking $\eta$ so small that $10C^2\eta<\frac{\theta}{4}$ we therefore get that
\begin{equation}
(\partial_s-\Delta)h \leq -2h+\frac{\theta}{4}.
\end{equation}
Thus, if the supremum of $h$ at time $s$ is larger than $\frac{\theta}{4}$ then  at time $s$ 
\begin{equation}
\frac{d}{ds}\max h(\cdot,s) \leq -\theta/4. 
\end{equation}
Taking $s_1\geq \frac{4C}{\theta}$ we get that for every $s\geq s_1$
\begin{equation}
h \leq \theta/4. 
\end{equation}
In particular, if $r\leq \theta$ and $s \geq s_1$ then 
\begin{equation}
(1-\theta^2)^2(v^2-1) \leq \frac{\theta}{4}
\end{equation}
so
\begin{equation}
v \leq \sqrt{1+ \frac{\theta}{2(1-\theta^2)^2}} \leq 1+\theta.  
\end{equation}

\end{proof}

\section{barriers}

All the theorems thus far had assumed that $t\geq T$ along the flow. The following barrier argument provides the mechanism that ensures this condition.
\begin{proposition}[Lower Barrier]\label{barrier_proposition}
There exists $R_2<\infty$ with the following significance: consider the MCF with boundary $\Sigma_s \subseteq \mathcal{N}_{R_2}$ such that  
\[
\Sigma_0=\{(x,0)\;:\; |x|\leq R_2\}
\] 
and such that 
\[
\Sigma_s\cap \{(x,t)\;:\; |x|= R_2\}= \{(x,0)\;:\; |x|=R_2\}.
\]
Then $\Sigma_s$ exists for every $s\geq 0$ and writing $\Sigma_s$ as a graph of a function $w(\cdot,s)$ we have   
\begin{equation}
\lim_{s\rightarrow \infty} w(0,s)=\infty.
\end{equation} 
\end{proposition}

\begin{proof}
First, observe that for every $\varepsilon>0$ the hypersurfaces $\bar{M}_{\varepsilon+s}$ form upper barriers to the flow $\Sigma_s$. In particular, for every $s$ we have 
\begin{equation}
0 \leq \inf_{|x| \leq R_2}w(x,s)<\sup_{|x|\leq R_2}w(x,s) \leq 3s.
\end{equation}
By standard parabolic theory, this implies that $\Sigma_s$ exists for all time. 
Now, denote by $\sigma_s$ the MCF evolution with boundary 
\[
\sigma_{ 0}=\{(x,0)\;:\; |x|\leq 1\}, \qquad \sigma_s\cap \{(x,t)\;:\; |x|= 1\}= \{(x,0)\;:\; |x|=1\},
\]
so that $\sigma_s$ coincides with $\Sigma_s$ if $R_2$ were one. By the same token as above $\sigma_s$ exists for every $s\in [0,\infty)$. Moreover,  since $\sigma_0$ is strictly mean convex, { and as the maximum principle implies that $\Sigma_s$ lies above $\sigma_s$ where both are defined,}  we have that {
\begin{equation}
w(0,s)\geq w(0,1)=c>0, 
\end{equation}
for every $s\geq 1$.
}
Using translates of $\sigma_s$ as lower barriers to $\Sigma_s$ we therefore obtain that
\begin{equation}\label{ux1low}
w(x,1)\geq c\qquad \textrm{if}\qquad|x|\leq R_2-1. 
\end{equation}
Recalling the isometries
\begin{equation}\label{O_a_recall}
O_{a}(x,t):=(e^a x,t-a),
\end{equation}
if $R_2$ is such that { $e^{c}(R_2-1)\geq R_2$}   then by virtue of \eqref{ux1low}, the evolution
\begin{equation}
\Sigma^1_s{:}=O_c(\Sigma_{1+s})\cap \mathcal{N}_{R_2}
\end{equation}
is a mean convex graphical evolution in $\mathcal{N}_{R_2}$, which is contained in the region where $t\geq 0$. In particular $\Sigma^1_s$ lies above $\Sigma_s$ or, put differently,  
\begin{equation}\label{a_little_up}
w(x,1+s) \geq w(e^{c}x,s)+c,
\end{equation}
if { $|x| \leq e^{-c}(R_2-1)$ } and $s\geq 0$.

Using \eqref{a_little_up} iteratively, we get that for every $j\in \mathbb{N}$
\begin{equation}
w(x,j+s) \geq w(e^{jc}x,s)+jc,
\end{equation}
if { $|x| \leq e^{-jc}(R_2-1)$} and $s\geq 0$, from which the result follows.       
\end{proof}

\section{proof of the main theorem}
We are now in the position to put everything together and conclude the proof of the main theorem from the introduction.
\begin{proof}[Proof of Theorem \ref{main_thm}]
{The proof consists of the following steps. We first employ Proposition \ref{barrier_proposition} to show that the flow reaches infinite height on a smaller de Sitter cylinder. We then use Theorem \ref{grad_est} (gradient estimate) to obtain some uniform gradient bounds. These bounds in turn are used as an input for Theorem \ref{flat_est} (flatness estimate), which are then integrated to obtain local derivative estimates. Once such estimates are obtained, the proof is concluded via standard parabolic estimates.}     

\bigskip

{For the first step}, observe that Proposition \ref{barrier_proposition} (lower barrier) implies that if { $R\geq \max\{2R_2,2\}$}, then setting $\psi(s)=w(0,s)$, we have 
\begin{equation}\label{psi_goes}
\psi(s)\nnearrow \infty.
\end{equation}
On the other hand, denoting
\begin{equation}
S_{x_0}(x,t)=(x+x_0,t),
\end{equation}
the assumption that $M_s$ lies in $t\geq 0$, { and our choice of $R\geq 2R_2$, imply} that $M_0$ lies above $S_{x_0}(\Sigma_0)$  for every $x_0$ with $|x_0| \leq {\frac{R}{2}}$, and { $M_s\cap S_{x_0}(\partial \mathcal{N}_{R_2})$ lies above $S_{x_0}\left(\partial \Sigma_s\right)$ for such an $x_0$ and every $s\geq 0$}. The maximum principle thus implies that for every $x$ with $|x|\leq R/2$ we have that 
\begin{equation}
\min_{z\in M_s\cap \mathcal{N}_{R/2}} t(z)\geq \psi(s).
\end{equation}
Renaming $R$ to be $R/2$ we may therefore assume w.l.g that
\begin{equation}\label{bd_up}
\min_{z\in M_s\cap \mathcal{N}_{R}} t(z)\geq \psi(s),
\end{equation}
{ and that $R\geq 1$.} 

{Having established \eqref{bd_up}, we turn to proving a uniform gradient estimate.} Denote 
\begin{equation}\label{Lambdadef}
\Lambda:=\sup_{s\geq 0}\;\sup_{z\in M_s}H(z)+{1 }
\end{equation}
which is finite, by assumption. {Let $\alpha_0$ and $C=C(\Lambda)$ be the constants from Theorem \ref{grad_est} (gradient estimate), let $\alpha<\alpha_0$ and consider the constants $T_0(\alpha)$ and $R_0(\alpha)$ from that same theorem. Since for every $a$
\begin{equation}
O_{a}\left(\mathcal{N}_R\right)=\mathcal{N}_{e^{a}R},
\end{equation}
{and as $O_a$ is an isometry}, using \eqref{bd_up} and \eqref{psi_goes} we see that there exists $a(\alpha)$ and $s_{\ast}(\alpha)$ such that the flow
\begin{equation}
\tilde{M}_s^{\alpha}:=O_{a(\alpha)}(M_{s_{\ast}(\alpha)+s})
\end{equation}
is a mean convex flow in $\mathcal{N}_{R(\alpha)}$ with $|H|\leq \Lambda$ satisfying $t(z)\geq T_0(\alpha)$ for every $s\geq 0$ and $z \in \tilde{M}_s^{\alpha}\cap D_{\alpha,R(\alpha)}$.  Applying Theorem { \ref{grad_est}} (gradient estimate) we therefore get that for $s\geq s_0(\alpha):=s_0(M_{s_{\ast}},\Lambda)$ we have
\begin{equation}
\sup_{z\in \tilde{M}_s^{\alpha}\cap D_{\alpha,R(\alpha)/{2}}} v\leq C. 
\end{equation} 
}

\bigskip

We are now going to improve the gradient estimates to flatness estimates. For that, consider $\theta>0$ and choose { $\alpha(C,\theta)$, $T_1(C,\theta)$ and $s_1(C,\theta)$}  as in Theorem \ref{flat_est} (flatness estimate). 
Applying \eqref{psi_goes} and \eqref{bd_up} once more, we get that there exists {$s_2=s_2(C,\theta)\geq s_0(\alpha(C,\theta))$} such that 
\begin{equation}
\min_{z\in \tilde{M}^{\alpha}_s\cap D_{\alpha({C},\theta),1}} t(z)\geq T_1({C},\theta),
\end{equation}
for every $s\geq s_2(C,\theta)$.  Therefore, Theorem \ref{flat_est} (flatness estimate) implies that
\begin{equation}\label{fin_grad_bd}
\sup_{z\in \tilde{M}^{\alpha}_s \cap D_{\alpha(C,\theta),\theta}} v \leq 1+\theta,
\end{equation}
for every $s\geq s_2(C,\theta)+s_1(C,\theta)$.  In terms of the original flow $M_s$ this implies that
\begin{equation}\label{fin_grad_bd_trans}
\sup_{z\in M_s \cap O_{-a(\alpha(C,\theta))}\left(D_{\alpha(C,\theta),\theta}\right)} v \leq 1+\theta,
\end{equation}
or every $s\geq s_3(C,\theta):=s_2(C,\theta)+s_1(C,\theta)+{s_{\ast}(\alpha(C,\theta))}$.

\bigskip

Let us now set
\begin{equation}
\varphi(\lambda):=u(0,\lambda) \geq \psi(\lambda),
\end{equation}
and 
\begin{equation}\label{Mslambda_def}
M_s^{\lambda}:=O_{\varphi(\lambda)}(M_{s+\lambda}),
\end{equation}
i.e we translate in time (using the isometry) so that $(0,0)\in M_0^{\lambda}$. {Note that $M_s^{\lambda}$ is the graph of the function $u^{\lambda}(\cdot,s)$  from the statement of the theorem.

In order to obtain the asserted convergence of the $u^{\lambda}$, we proceed in two steps. We first show that $M^{\lambda}_s$ is a graph with small gradient over larger and larger balls in the $x_i$ co-ordinates for larger and larger flow time intervals, \textit{provided} the height of that graph remains quantitatively  bounded. We then use that $(0,0)\in M_{0}^{\lambda}$ and the above conditional bounds to integrate the height on growing balls, supplying the assumption that is needed for the conditional result.}

For the first step, note that since $\alpha(C,\theta)<2$, for every $\rho<\infty$, there exists $h_0$ such that for every $h\geq h_0$
\begin{equation}\label{incl_dom_first}
E_{\rho}:=\{(x,t)\;|\; |x|\leq \rho,\; |t| \leq \rho \} \subseteq O_{h}\left(D_{\alpha(C,\theta),\theta}\right).
\end{equation}
In particular, as  $\varphi(\lambda)\nnearrow \infty$, for every $\theta\in (0,1)$ and $\rho=\rho(\theta)<\infty$ to be determined later, we have that
\begin{equation}\label{incl_dom}
E_{\rho} \subseteq O_{ \varphi(\lambda)-a(\alpha(C,\theta))}\left(D_{\alpha(C,\theta),\theta}\right), 
\end{equation}
provided $\lambda\geq \lambda_0(\theta,\rho)$.  Therefore, if $\lambda\geq \lambda_0(\theta,\rho)$ and if $\lambda-\log(\rho)\geq s_3({C},\theta)$ {then combining \eqref{fin_grad_bd_trans} and \eqref{incl_dom}} we get
\begin{equation}\label{fin_grad_bd_trans2}
\sup_{z\in M^{\lambda}_s \cap E_{\rho}} v \leq 1+\theta,
\end{equation}
for every $s\in [-\log(\rho),\log(\rho)]$, where we stress that $M_s^{\lambda}$ is defined by \eqref{Mslambda_def} (to avoid confusion with the auxiliary $\tilde{M}_s^{\alpha}$ which will not take a further role in this proof).  {Note that \eqref{fin_grad_bd_trans2} only gives gradient bounds at points where  $|u^{\lambda}| \leq \rho$, which is the  nature of the ``conditional'' gradient bounds of the first step which was described above.}

{For the second step, recall that by }construction $(0,0)\in M^{\lambda}_0$. Combining 
\[
\varphi'(s)=H(0,s)v(0,s),
\]
\eqref{fin_grad_bd_trans2}, \eqref{Lambdadef}, mean convexity and our choice of $\theta<1$ we get  that 
\begin{equation}\label{center_bd}
|u^{\lambda}(0,s)| \leq 2\Lambda s,\;\;\; \textrm{provided}\;\;\;s\in \left[\frac{-\log(\rho)}{8\Lambda}, \frac{\log(\rho)}{8\Lambda} \right].
\end{equation}
{We will use \eqref{center_bd} and \eqref{fin_grad_bd_trans2} to integrate the height bound, which will confirm that $|u^{\lambda}(\cdot,s)| \leq \rho$ on large balls in the $x$ domain. To that end,} taking any $\omega \in \mathbb{S}^2(1)$, { note that $y\mapsto (y\omega,u^{\lambda}(y\omega,s))$ is a curve in $M^{\lambda}_s$ with velocity 
\begin{equation}
\frac{d}{dy}(y\omega,u^{\lambda}(y\omega,s))=(\omega,\omega\cdot Du^{\lambda}(-,s)),
\end{equation}
{where $Du^{\lambda}$ is the Euclidean derivative of $u^{\lambda}(-,s)$.} Therefore
\begin{equation}\label{motion_bd_a}
\left|\frac{d}{dy}u^{\lambda}(y\omega,s)\right| \leq |\nabla u^{\lambda}(y\omega,s)|_g \left|(\omega, \omega \cdot Du^{\lambda})\right|_{g}.
\end{equation}
By the definition of the metric $g$, if $|u^{\lambda}(y\omega,s)| \leq \frac{\log(\rho)}{2}$ then 
\begin{equation}\label{motion_bd_b}
\left|(\omega, \omega \cdot Du^{\lambda}(y\omega,s))\right|_g \leq e^{u^{\lambda}(y\omega,s)}\leq \sqrt{\rho}.
\end{equation}
}
Additionally,  by virtue  of \eqref{fin_grad_bd_trans2} and \eqref{grad_sq}, if $y\leq \rho$ and $|u^{\lambda}(y\omega,s)| \leq \frac{\log(\rho)}{2}$ then  
\begin{equation}\label{motion_bd_c}
|\nabla u^{\lambda}(y\omega,s)|_g \leq \sqrt{v^2-1} \leq 2\sqrt{\theta}.
\end{equation}
Putting \eqref{motion_bd_a}, \eqref{motion_bd_b} and \eqref{motion_bd_c} together we therefore get that if $y\leq \rho$ and $|u^{\lambda}(y\omega,s)| \leq \frac{\log(\rho)}{2}$ then 
\begin{equation}\label{motion_bd}
\left|\frac{d}{dy}u^{\lambda}(y\omega,s)\right| \leq 2\sqrt{\theta}e^{u^{\lambda}(y\omega,s)} \leq 2\sqrt{\theta \rho},
\end{equation}
Therefore, defining 
\begin{equation}\label{rho_def}
\rho=\rho(\theta)=\theta^{-\frac{1}{3}},
\end{equation}
which blows up as $\theta\rightarrow 0$, then {using  \eqref{center_bd} and integrating  \eqref{motion_bd} for as long as it is valid, } we get that if $\lambda$ is sufficiently large, { and $\theta \leq \frac{1}{1000}$,} then
\begin{equation}
|u^{\lambda}(x,s)|\leq \frac{\log(\rho(\theta))}{2},\;\;\; \textrm{provided}\;\;\; |x|\leq \rho(\theta),\;\;s\in  \left[\frac{-\log(\rho(\theta))}{8\Lambda}, \frac{\log(\rho(\theta))}{8\Lambda} \right].
\end{equation}
Thus, for the above values of $s$ we have $M^{\lambda}_s\cap \mathcal{N}_{\rho}=M^{\lambda}_s \cap E_{\rho}$, so \eqref{fin_grad_bd_trans2} implies

\begin{equation}\label{fin_grad_bd_trans3}
\sup_{z\in M^{\lambda}_s \cap \mathcal{N}_{\theta^{-1/3}}} v \leq 1+\theta,\;\;s\in  \left[\frac{-\log(\theta^{-1/3})}{8\Lambda}, \frac{\log(\theta^{-1/3})}{8\Lambda} \right], \; \lambda\; \textrm{is sufficiently large}.
\end{equation}

\bigskip

Finally, using \eqref{fin_grad_bd_trans3} and \eqref{center_bd} and standard parabolic theory \cite{Lieberman}, we get uniform higher derivative estimates for $u^\lambda$ on compact subsets of space time. In particular, using Arzela Ascoli, \eqref{fin_grad_bd_trans3} and \eqref{center_bd} again, we get that
\begin{equation}
u^{\lambda}\xrightarrow{C^{\infty}_{\mathrm{loc}}(\mathbb{R}^3\times \mathbb{R})} f(s),
\end{equation}
for some smooth function $f:\mathbb{R}\rightarrow \mathbb{R}$. Since $f(0)=0$ by construction and since each slice $\{t=t_0\}$ in de Sitter space has constant mean curvature $3$, it follows that $f(s)=3s$. This concludes the proof of the main theorem. 
\end{proof}

\bigskip

\bibliography{deSitter}

\begin{thebibliography}{CHSV20}

\bibitem[Bar84]{Bartnik}
Robert Bartnik.
\newblock Existence of maximal surfaces in asymptotically flat spacetimes.
\newblock {\em Comm. Math. Phys.}, 94(2):155--175, 1984.

\bibitem[BS03]{BS}
Antonio~N. Bernal and Miguel S\'{a}nchez.
\newblock On smooth {C}auchy hypersurfaces and {G}eroch's splitting theorem.
\newblock {\em Comm. Math. Phys.}, 243(3):461--470, 2003.

\bibitem[CHSV20]{CHSV}
Paolo Creminelli, Or~Hershkovits, Leonardo Senatore, and András Vasy.
\newblock A de sitter no-hair theorem for 3+1d cosmologies with isometry group
  forming 2-dimensional orbits, 2020.

\bibitem[CSV20]{CSV}
Paolo Creminelli, Leonardo Senatore, and Andr\'{a}s Vasy.
\newblock Asymptotic behavior of cosmologies with {$\Lambda>0$} in {$2+1$}
  dimensions.
\newblock {\em Comm. Math. Phys.}, 376(2):1155--1170, 2020.

\bibitem[Eck93]{Ecker_flat}
Klaus Ecker.
\newblock On mean curvature flow of spacelike hypersurfaces in asymptotically
  flat spacetimes.
\newblock {\em J. Austral. Math. Soc. Ser. A}, 55(1):41--59, 1993.

\bibitem[Eck97a]{Ecker_null}
Klaus Ecker.
\newblock Interior estimates and longtime solutions for mean curvature flow of
  noncompact spacelike hypersurfaces in {M}inkowski space.
\newblock {\em J. Differential Geom.}, 46(3):481--498, 1997.

\bibitem[Eck97b]{Ecker_mink}
Klaus Ecker.
\newblock Interior estimates and longtime solutions for mean curvature flow of
  noncompact spacelike hypersurfaces in {M}inkowski space.
\newblock {\em J. Differential Geom.}, 46(3):481--498, 1997.

\bibitem[EH89]{EH_entire}
Klaus Ecker and Gerhard Huisken.
\newblock Mean curvature evolution of entire graphs.
\newblock {\em Ann. of Math. (2)}, 130(3):453--471, 1989.

\bibitem[EH91a]{EH_interior}
Klaus Ecker and Gerhard Huisken.
\newblock Interior estimates for hypersurfaces moving by mean curvature.
\newblock {\em Invent. Math.}, 105(3):547--569, 1991.

\bibitem[EH91b]{Ecker_Huisken_lor}
Klaus Ecker and Gerhard Huisken.
\newblock Parabolic methods for the construction of spacelike slices of
  prescribed mean curvature in cosmological spacetimes.
\newblock {\em Comm. Math. Phys.}, 135(3):595--613, 1991.

\bibitem[Ger70]{Geroch}
Robert Geroch.
\newblock Domain of dependence.
\newblock {\em J. Mathematical Phys.}, 11:437--449, 1970.

\bibitem[GH77]{GibonHawking}
G.~W. Gibbons and S.~W. Hawking.
\newblock Cosmological event horizons, thermodynamics, and particle creation.
\newblock {\em Phys. Rev. D}, 15:2738--2751, May 1977.

\bibitem[HE73]{Hawking_Ellis}
S.~W. Hawking and G.~F.~R. Ellis.
\newblock {\em The large scale structure of space-time}.
\newblock Cambridge Monographs on Mathematical Physics, No. 1. Cambridge
  University Press, London-New York, 1973.

\bibitem[HM82]{HawkingMoss}
S.W. Hawking and I.L. Moss.
\newblock Supercooled phase transitions in the very early universe.
\newblock {\em Physics Letters B}, 110(1):35--38, 1982.

\bibitem[KS16]{KS}
Matthew Kleban and Leonardo Senatore.
\newblock Inhomogeneous anisotropic cosmology.
\newblock {\em Journal of Cosmology and Astroparticle Physics}, 2016(10):022,
  oct 2016.

\bibitem[Lam14]{Lamb_neu}
Ben Lambert.
\newblock The perpendicular {N}eumann problem for mean curvature flow with a
  timelike cone boundary condition.
\newblock {\em Trans. Amer. Math. Soc.}, 366(7):3373--3388, 2014.

\bibitem[Lie96]{Lieberman}
Gary~M. Lieberman.
\newblock {\em Second order parabolic differential equations}.
\newblock World Scientific Publishing Co., Inc., River Edge, NJ, 1996.

\bibitem[RS22]{RS}
Henri Roesch and Julian Scheuer.
\newblock Mean curvature flow in null hypersurfaces and the detection of
  {MOTS}.
\newblock {\em Comm. Math. Phys.}, 390(3):1149--1173, 2022.

\bibitem[Smo13]{Smock_evol}
Knut Smoczyk.
\newblock Evolution of spacelike surfaces in ads3 by their {L}agrangian angle.
\newblock {\em Math. Ann.}, 355(4):1443--1468, 2013.

\end{thebibliography}

\bibliographystyle{alpha}

\vspace{10mm}

{\sc Or Hershkovits, Institute of Mathematics, Hebrew University, Givat Ram, Jerusalem, 91904, Israel}\\

{\sc Leonardo Senatore, Department of Physics, ETH Zurich, , Switzerland}\\

\emph{E-mail:}  or.hershkovits@mail.huji.ac.il, lsenatore@ethz.ch

\end{document}